\documentclass[10pt,reqno]{amsart}
\usepackage[utf8]{inputenc}
\usepackage{amsopn,amssymb,mathrsfs,xcolor,url,enumitem,accents}
\usepackage[abbrev]{amsrefs} 
\usepackage[a4paper,lmargin=3cm,rmargin=3cm,tmargin=4cm,bmargin=4cm,marginparwidth=2.8cm,marginparsep=1mm]{geometry} 
\usepackage[bookmarks=true,hyperindex,pdftex,colorlinks,citecolor=blue]{hyperref} 
\usepackage{cleveref}

\usepackage{marginnote} 

\newcommand{\restrict}{\mathord{\upharpoonright}}           

\newtheorem{theorem}{Theorem}[section]
\newtheorem{corollary}[theorem]{Corollary}
\newtheorem{lemma}[theorem]{Lemma}
\newtheorem{proposition}[theorem]{Proposition}
\newtheorem{problem}[theorem]{Problem}

\theoremstyle{definition}

\newtheorem{remark}[theorem]{Remark}

\newcommand{\nm}[1]{\vert\kern-0.25ex\vert #1 \vert\kern-0.25ex\vert}

\newcommand{\R}{\mathbb{R}}
\newcommand{\N}{\mathbb{N}}
\newcommand{\diam}{\textnormal{diam}}
\newcommand{\free}[1]{\mathcal{F}(#1)}
\newcommand{\Lipo}[1]{\textnormal{Lip}_0(#1)}

\newcommand{\cM}{\mathcal{M}}

\newcommand{\cP}{\mathcal{P}}
\newcommand{\cA}{\mathcal{A}}

\DeclareMathOperator{\Lip}{Lip}

\numberwithin{equation}{section}

\allowdisplaybreaks

\title{Lipschitz-free spaces over properly metrisable spaces and approximation properties}
\author[R. J. Smith]{Richard J.~Smith}
\address[R. J. Smith]{School of Mathematics and Statistics, University College Dublin, Belfield, Dublin 4, Ireland}
\email{richard.smith@maths.ucd.ie}
\author[F. Talimdjioski]{Filip Talimdjioski}
\address[F. Talimdjioski]{School of Mathematics and Statistics, University College Dublin, Belfield, Dublin 4, Ireland}
\email{filip.talimdjioski@ucdconnect.ie}

\date{\today}

\begin{document}
	\begin{abstract}
		Let $T$ be a topological space admitting a compatible proper metric, that is, a locally compact, separable and metrisable space. Let $\cM^T$ be the non-empty set of all proper metrics $d$ on $T$ compatible with its topology, and equip $\cM^T$ with the topology of uniform convergence, where the metrics are regarded as functions on $T^2$. We prove that the set $\cA^{T,1}$ of metrics $d\in\cM^T$ for which the Lipschitz-free space $\free{T,d}$ has the metric approximation property is a dense set in $\cM^T$, and is furthermore residual in $\cM^T$ when $T$ is zero-dimensional. We also prove that if $T$ is uncountable then the set $\cA^T_f$ of metrics $d\in\cM^T$ for which $\free{T,d}$ fails the approximation property is dense in $\cM^T$. Combining the last statement with a result of Dalet, we conclude that for any `properly metrisable' space $T$, $\cA^T_f$ is either empty or dense in $\cM^T$.
	\end{abstract}
	
	\maketitle

	\section{Introduction}\label{sect_intro}
	
	For a metric space $(M,d)$ and a base point $x_0\in M$ we define the Banach space $\Lipo{M,x_0}$ consisting of all real-valued Lipschitz functions $f$ on $M$ that vanish at $x_0$, equipped with the norm $$\nm{f} := \sup_{x,y\in M, x\not= y} \frac{|f(x)-f(y)|}{d(x,y)}.$$ For any $x\in M$ we define the bounded linear functional $\delta_x\in \Lipo{M,x_0}^*$ by $\delta_x(f) = f(x)$, $f\in\Lipo{M,x_0}$. The closed linear span of the set $\{\delta_x : x\in M\}$ is called the \emph{Lipschitz-free space} $\free{M,x_0}$. It is well-known that $\Lipo{M,x_0}$ is isometric to the dual space of $\free{M,x_0}$ and that the isometric structure of both spaces does not depend on the choice of the base point $x_0$. We will hereafter write $\Lipo{M}$ and $\free{M}$ without specifying the base point, and call $\free{M}$ simply the free space over $M$. In the book \cite{weaver}, Weaver provides a comprehensive introduction to Lipschitz and Lipschitz-free spaces. In it, the latter are called Arens-Eells spaces and are denoted by \AE$(M)$.
	
	The investigation of the structure of Lipschitz-free spaces is an area of active ongoing research in Banach space theory. One direction of research has been the approximation properties of free spaces. Results involving the approximation property (AP) appear in \cites{godefroyozawa, hajeklancienpernecka, kalton}; results involving the bounded approximation property (BAP) appear in \cites{godefroykalton, godefroy:15, apandschd}; and metric approximation property (MAP) results can be found in in \cites{godefroykalton, adalet, godefroyozawa, douchakaufmann, cuthdoucha,fonfw,ps:15,smithtalim}. More details are available in the introduction to \cite{smithtalim}, and in \cite{godefroy:20}.
	
	We will mention several results regarding compact and proper metric spaces (the latter property means that every closed ball is compact). It is known that if $M$ is a sufficiently `thin' totally disconnected metric space then $\free{M}$ has the MAP. For example, if $M$ is a countable proper metric space then $\free{M}$ has the MAP \cite{adalet}. Also, as a corollary of \cite{weaver}*{Corollary 4.39} and \cite{cuthdoucha}*{Corollary 16}, $\free{M}$ has the MAP when $M$ is compact and uniformly disconnected, meaning that there exists $0 < a \leq 1$ such that for every distinct $p,q\in M$, we can find complementary clopen sets $C,D \subseteq M$ such that $p\in C, q\in D$ and $d(C,D) \geq a d(p,q)$, where $d(C,D) = \inf \{d(x,y) : x\in C, y\in D\}$. For example, the usual middle-thirds Cantor subset of $[0,1]$ has this property. Moreover, as a corollary of \cite{apandschd}*{Proposition 2.2} and \cite{purely1unrect}*{Theorem B}, $\free{M}$ has the MAP when $M$ is doubling and purely 1-unrectifiable, the latter property being equivalent to the condition that $M$ contains no bi-Lipschitz image of a compact subset of $\R$ of positive measure. On the other hand, in \cite{godefroyozawa}, G.~Godefroy and N.~Ozawa constructed a compact convex subset $C$ of a separable Banach space such that $\free{C}$ fails the AP. Also, by a result in \cite{hajeklancienpernecka}, there exists a metric space $M$ homeomorphic to the Cantor space such that $\free{M}$ fails the AP. The recent result \cite{douchakaufmann} shows that free spaces over compact groups equipped with a left-invariant metric have the MAP. A characterisation of the BAP on free spaces over compact metric spaces, in terms of uniformly bounded sequences of Lipschitz `almost extension' operators, is given in \cite{godefroy:15}.
	
	In \cite{godefroysurvey}, G.~Godefroy surveys various aspects of the theory of free spaces, including the lifting property for separable Banach spaces, approximation properties of free spaces, and norm attainment of Lipschitz functions and operators. In the last section he states a number of open problems, several of which concern approximation properties of free spaces.
	
	For a topological space $T$, we denote by $\cM^T$ the set of all proper metrics on $T$ compatible with its topology. Note that every proper metric is complete. By \cite{engelking}*{Exercise 4.2.C} $\cM^T\not=\emptyset$, i.e.~$T$ is `properly metrisable', if and only if $T$ is locally compact, separable and metrisable. Hereafter we suppose that $T$ always belongs to this class of topological spaces. Note that if $T$ is compact and metrisable then $\cM^T$ is the set of all metrics on $T$ compatible with its topology. Let us consider the vector space $C(T^2)$ of continuous (possibly unbounded) real-valued functions on $T^2$. We equip $C(T^2)$ with the topology of uniform convergence, which according to \cite{engelking}*{Theorem 4.2.20} is induced by the metric $\rho_T$ given by
		\[
		\rho_T(f,g) = \sup_{(x,y)\in T^2} \min(1,|f(x,y)-g(x,y)|), \quad f, g \in C(T^2).
		\]
	It can be shown that $\cM^T$, equipped with the subspace topology inherited from $(C(T^2), \rho_T)$, is a completely metrisable space, and hence a Baire space (a proof is given in the next section). We note that the compact-open topology, which in this context is the topology of uniform convergence on compact sets, is another natural candidate with which to endow $\cM^T$. Indeed, $C(T^2)$ endowed with the compact-open topology is a Polish space \cite{engelking}*{3.4.H (a) and 4.3.F (a)}. However, $\cM^T$ with this topology is a Baire space if and only if $T$ is compact \cite{koshino2022}*{Proposition 3.2}. For this reason we prefer to work with the topology of uniform convergence; evidently, if $T$ is compact then the two topologies agree.
		
	Let us define the sets
	\begin{align*}
		\cA^{T,\lambda} &= \{d\in\cM^T \; :\; \free{T,d} \text{ has the } \lambda\text{-BAP}\}, \text{ for } \lambda \in [1,\infty), \\
		\text{and}\quad\cA^T_f &= \{d\in\cM^T \; :\; \free{T,d} \text{ fails the AP} \}.
	\end{align*}
	
	In \cite{godefroysurvey}*{Problem 6.6}, Godefroy considered the set $\cM^K$, where $K=2^\N$ is the Cantor space equipped with the standard (product) topology. He asked about the topological nature of the set $\cA_f^K$ and, in particular, whether $\cA_f^K$ is a residual set in $\cM^K$. The latter question has a negative answer because $\cA^{K,1}$ is a residual ($F_{{\sigma\delta}}$) set in $\cM^K$ \cite{talimpublished}*{Theorem 1.1}. In addition, it was shown that the necessarily meager set $\cA^K_f$ is dense in $\cM^K$ \cite{talimpublished}*{Proposition 1.2}. In this paper we consider Godefroy's question with respect to properly metrisable spaces $T$. In our opinion this question is natural to consider, given that proper metric spaces include all compact metric spaces and closed subsets of $\R^n$. Because they form a wide class of relatively well-behaved metric spaces, they are an important special case in Lipschitz-free space theory; see e.g.~\cite{adalet}. Related to this is \cite{godefroysurvey}*{Problem 6.4}, which asks whether $\free{M}$ has the MAP whenever $(M,d)$ is a (closed) subset of $\R^n$ and $d$ is induced by a norm on $\R^n$. Our main results, which we list below, generalise those of \cite{talimpublished}.
	
	\begin{theorem}\label{thm:mapdense}
		The set $\cA^{T,1}$ is dense in $\cM^T$.
	\end{theorem}
	
	\Cref{thm:mapdense} generalizes the fact that, by \cite{ishiki2021dense}*{Theorem 1.1} and \cite{apandschd}*{Proposition 2.2}, $\bigcup_{n=1}^\infty \cA^{T,n}$ is dense in $\cM^T$ whenever $T$ is compact and has finite covering dimension.
	
	\begin{theorem}\label{thm:failapdense}
		If $T$ is uncountable, then $\cA^T_f$ is dense in $\cM^T$.
	\end{theorem}
	
	\begin{proposition}\label{prop:residual}
			If $T$ is zero-dimensional then $\cA^{T,1}$ is residual in $\cM^T$. Consequently, $\cA^T_f$ is meager in $\cM^T$.
	\end{proposition}
	
	Finally, we combine \Cref{thm:failapdense} with one in \cite{adalet} to arrive at a dichotomy.
	
	\begin{corollary}
		Let $T$ be a locally compact, separable, and metrisable space. Then either $T$ is countable and $\cA^{T,1} = \cM^T$, or $T$ is uncountable and $\cA^T_f$ is dense in $\cM^T$.
	\end{corollary}
	\begin{proof}
		The first half of the statement follows from \cite{adalet}*{Theorem 2.6}, and the second half from \Cref{thm:failapdense}.
	\end{proof}
	
	Again by \cite{adalet}*{Theorem 2.6}, we observe that \Cref{thm:mapdense} provides new information when $T$ is uncountable.
	
	The rest of the paper is organised as follows. Section \ref{sect_anc} concerns notation and preliminary results, and \Cref{sec_main} is devoted to the proofs of \Cref{thm:mapdense,thm:failapdense} and \Cref{prop:residual}, and ends with an open problem.
	
	\section{Notation and preliminary results}\label{sect_anc}
	
	For a metric space $(M,d)$, $x\in M$, and $r>0$ we write
	\[
	B^{d,o}_r(x) = \{y\in M : d(x,y) < r\}, \quad\text{and}\quad B^d_r(x) = \{y\in M : d(x,y) \leq r\}. 
	\]
	If there is no ambiguity about the metric $d$ we write $B^o_r(x)$ and $B_r(x)$. If $A,B\subseteq M$ then write
	\begin{align*}
		d(A,B) = \inf_{x\in A,y\in B} d(x,y),&\quad d(x,A) = d(\{x\},A) \\ D(A,B) = \sup_{x\in A,y\in B} d(x,y),&\quad D(x,A) = D(\{x\},A).
	\end{align*}
	If $X$ is a Banach space then $B_X$ denotes the closed unit ball of $X$.
	
	For a real-valued Lipschitz function $f$ on $(M,d)$, we write $\Lip_d(f)$ for the (optimal) Lipschitz constant of $f$ with respect to $d$. The Lipschitz space and free space over $(M,d)$ are denoted by $\Lipo{M,d}$ and $\free{M,d}$, respectively. If there is no ambiguity about the metric $d$ we omit it from the notation. For a bounded real-valued function $g$ defined on a set $U$, we write $\nm g _\infty = \sup_{x\in U} |g(x)|$. If $A\subseteq U$ is a subset then we write $g\restrict_A$ for the restriction of $g$ to $A$. We call two metric spaces $(A,d)$ and $(B,e)$ \emph{proportional} if there exists a surjection $f\colon A\to B$ and $c>0$ such that $d(x,y) = ce(f(x),f(y))$ for all $x,y\in A$. For a non-empty set $S$, we call a finite family $\{S_1,\ldots,S_n\}$ of non-empty subsets of $S$ a \emph{partition} of $S$ if $S_i\cap S_j = \emptyset$ whenever $i\not=j$ and $\bigcup_{i=1}^n S_i = S$.
	
	Throughout the rest of the paper $T$ will be a fixed locally compact, separable, and metrisable topological space with base point $x_0\in T$. It will help to fix a metric $\alpha \in \cM^T$. For $r>0$ and $d\in\cM^T$ we write $B_r^d$ and $B_r^{d,o}$ for $B_r^d(x_0)$ and $B_r^{d,o}(x_0)$, respectively. We will first prove that $\cM^T$, equipped with the topology of uniform convergence, is completely metrisable. Note that by \cite{engelking}*{Theorem 4.3.13}, $(C(T^2),\rho_T)$ is complete, since the metric $(x,y)\mapsto \min(1,|x-y|)$ on $\R$ is bounded and complete. We remark that the result below can be proved using \cite{ishiki2020interpolation}*{Lemma 5.1}; however, we use an argument in the proof below later in the proof of \Cref{lm:closedboundedcompl}, so we prefer to keep things as they are.
	
	\begin{proposition}[{cf.~\cite{ishiki2020interpolation}*{Lemma 5.1}}]\label{lm:cmtisgdelta}
		The set $\cM^T$ is a $G_\delta$ subset of $(C(T^2),\rho_T)$ and thus completely metrisable.
	\end{proposition}
	\begin{proof}
		Let
		\begin{align*}
			A = \{f\in C(T^2) \; : \; f(x,x) &= 	f(x,y)-f(y,x) = 0, \text{ and } \\ f(x,z)&\leq f(x,y)+f(y,z) \text{ for all } x,y,z \in T\},
		\end{align*}
		and for $n\in\N$, define $$E_n = \{f\in C(T^2) \; :\; f(x,y) > 0 \text{ whenever } x,y\in B^\alpha_n, \alpha(x,y) \geq n^{-1}\}.$$ The set $A$ is closed, and hence $G_\delta$ in $C(T^2)$. Moreover, $E_n$ is open by the compactness of $B^\alpha_n$ for each $n\in\N$. Therefore the set $\cM^* = A\cap \bigcap_{n=1}^\infty E_n$ is $G_\delta$ and consists of all metrics on $T$ which are continuous as functions on $T^2$, but not necessarily proper or compatible with the topology on $T$. We claim that $\cM^T$ is an open subset of $\cM^*$. Pick $d\in \cM^T$ and suppose $d'\in \cM^*$ is such that $\rho_T(d,d') < 1$. By the continuity of $d'$ on $T^2$, the topology on $T$ induced by $d'$ is coarser than the original one. This implies that each compact subset of $T$ is also $d'$-compact. For each $n\in\N$, $B_n^{d'}$ is $d'$-closed, and hence $d$-closed. Also, $B_n^{d'} \subseteq B_{n+1}^{d}$, so $B_n^{d'}$ is $d$-compact, and hence $d'$-compact. Therefore $d'$ is proper. To show that $d'$ induces the original topology on $T$, suppose for a contradiction that $(y_n)_n\subseteq T$ is such that $d'(y_n,y)\to 0$ for some $y\in T$, but $d(y_n,y)\not\to 0$. Then there is $\epsilon > 0$ and a subsequence $(y_{n_k})_k$ such that $d(y_{n_k}, y) > \epsilon$ for all $k\in\N$. As $(y_{n_k})_k$ is $d'$-bounded, and hence $d$-bounded, there is a further subsequence $(y_{n_{k_i}})_i$ such that $d(y_{n_{k_i}}, z)\to 0$ for some $z\in T$, $z\not=y$. By the continuity of $d'$, $d'(y_{n_{k_i}}, z)\to 0$ and so $z=y$, a contradiction. Thus $d'\in \cM^T$. We conclude that $\cM^T$ is open as a subset of $\cM^*$, and so $\cM^T$ is a $G_\delta$ subset of $C(T^2)$.
	\end{proof}
	\begin{remark}\label{rmk:rhotmtdist}
		Note that $\rho_T(\cM^T, \cM^*\setminus\cM^T) = 1$, by the proof of the proposition.
	\end{remark}
	We therefore obtain that $\cM^T$ is a Baire space. Note that it may happen that $\cM^*\setminus \cM^T\not=\emptyset$; for example if $T=\R$ with its usual topology, then any compatible bounded metric on $\R$ belongs to $\cM^*\setminus \cM^T$. 
	
	We will need a criterion to determine when a linear operator between $\Lip_0$ spaces is a dual operator. Note that the $w^*$-topology on bounded subsets of $\Lipo{M}$ coincides with the topology of pointwise convergence \cite{weaver}*{Theorem 2.37}. Also note that if $M$ is separable then $\free{M}$ is separable. The next result can be deduced easily using an argument in the proof of \cite{apandschd}*{Theorem 2.1}, together with some standard results in Banach space theory.
	
	\begin{lemma}\label{lm:dualcriterion}
		Let $M$ and $N$ be complete metric spaces with base points. A bounded operator $S\colon \Lipo{M}\to\Lipo{N}$ is a dual operator if and only if $S(f_\lambda)\to S(f)$ pointwise whenever $(f_\lambda)_\lambda\subseteq B_{\Lipo{M}}$ is a net converging pointwise to $f\in B_{\Lipo{M}}$. If $M$ is separable then it is enough to consider sequences instead of nets.
	\end{lemma}
	
	We note that if $(M,d)$ is a complete metric space and $A\subseteq M$ is closed then $\free{A}$ can be identified naturally with a subspace of $\free{M}$ \cite{weaver}*{Theorem 3.7}. The following lemma will be needed in the proof of the main result.
	
	\begin{lemma}\label{lm:closedboundedcompl}
		Let $(M,d)$ be a complete metric space, and suppose $A\subseteq M$ is closed, bounded, and satisfies $d(A, M\setminus A) > 0$. Then $\free{A}$ is complemented in $\free{M}$.
	\end{lemma}
	\begin{proof}
		Pick a common base point in $A$ for both $\Lipo{A}$ and $\Lipo{M}$. We define the operator $G\colon \Lipo{A}\to \Lipo{M}$ by
		\begin{equation*}
			G(f)(x) = 
			\begin{cases}
				f(x) & \text{ if } x\in A, \\
				0 & \text{ otherwise. } 
			\end{cases}
		\end{equation*}
		Let $f\in B_{\Lipo{A}}$, $x\in A$, and $y\in M\setminus A$. We have $$|G(f)(x)-G(f)(y)| = |f(x)| \leq D(A) \leq \frac{D(A)}{d(A,M\setminus A)} d(x,y),$$ which implies that $\Lip(G(f)) \leq \max\left(1,\frac{D(A)}{d(A,M\setminus A)}\right)$. Clearly, if $(f_\lambda)_\lambda\subseteq B_{\Lipo{A}}$ converges pointwise to $f\in B_{\Lipo{A}}$, then $(G(f_\lambda))_\lambda$ converges pointwise to $G(f)$. By \Cref{lm:dualcriterion}, $G$ is a dual operator, and it is easily seen that its predual is a bounded projection of $\free{M}$ onto $\free{A}$. Therefore $\free{A}$ is complemented in $\free{M}$.
	\end{proof}
	
	If $A\subseteq T$ is a non-empty closed subset, $d_A$ is a (pseudo)metric on $A$, and $d_T$ is a (pseudo)metric on $T$, we say that $d_T$ \emph{extends} $d_A$ if $d_T\restrict_{A^2} = d_A$. We call a linear operator $E\colon\Lipo{A}\to\Lipo{T}$ an \emph{extension operator} if $T(f)\restrict_A = f$ for all $f\in \Lipo{A}$. \Cref{lm:extensionmetric} below enables us to extend a metric defined on $A$ to a metric on $T$ in a way that allows for the existance of a suitable extension operator. Its proof relies on an extension theorem by Dugundji \cite{dugundji}, which has been used in \cite{torunczyk1972simple} again for the purpose of extending metrics from a closed subset of a metrisable space to the whole space.
	
	We set the scene for this lemma by following \cite{dugundji}. Fix a non-empty closed set $A \subseteq T$. By A.~H.~Stone's theorem and the Lindel\"of property, we can find points $x_i$, $i \in \N$, and open neighbourhoods $U_i$ of $x_i$ satisfying $U_i \subseteq B_{r_i}^{d,o}(x_i) \subseteq T\setminus A$, where $r_i=\frac{d(x_i,A)}{3}$, such that $(U_i)_{i\in \N}$ is a locally finite cover of $T\setminus A$. We associate with the cover the usual partition of unity $\lambda_i:T\setminus A \to \R$, $i \in \N$, by setting $$\lambda_i(x) = \frac{d(x,U_i^c)}{\sum_{j\in \N} d(x,U_j^c)}.$$ As $(U_j)_j$ is a locally finite cover of $T\setminus A$, the sum in the denominator is finite and non-zero, so $\lambda_i$ is well-defined, continuous and vanishes outside $U_i$. As $d$ is proper, for each $i$ we can fix $a_i\in A$ such that $d(x_i,A) = d(x_i,a_i)$. Then we can define the map $F\colon C(A)\to C(T)$,
	\begin{equation*}
		F(f)(x) = 
		\begin{cases}
			f(x) & \text{ if } x\in A, \\
			\sum_{i\in \N} \lambda_i(x)f(a_i) & \text{ otherwise.}
		\end{cases}
	\end{equation*}
	Again, by local finiteness, $F(f)$ is well-defined and continuous on $T\setminus A$. That $F(f)$ is continuous on $T$ follows from \cite{dugundji}*{4.3}; for completeness we include an argument for this using the next brief lemma, which we will use again later.
	
	\begin{lemma}\label{lm:continuousextension}
		Let $a \in A$ and $\delta>0$. Then $d(a,a_i)<3\delta$ whenever $d(x,a)<\delta$ and $x \in U_i$.
	\end{lemma}
	
	\begin{proof}
		Let $d(x,a) < \delta$ and $x \in U_i$. As $d(\cdot,A)$ is $1$-Lipschitz, $d(x_i,A)-d(x,A) \leq d(x,x_i) < \frac{1}{3}d(x_i,A)$, so $d(x_i,A) < \frac{3}{2}d(x,A)$. Hence
		\[
		d(a,a_i) \leq d(a,x) + d(x,x_i) + d(x_i,a_i) < \delta + \frac{4}{3}d(x_i,A) < \delta + 2d(x,A) \leq 3\delta. \qedhere
		\]
	\end{proof}
	
	To see that $F(f)$ is continuous at all points of $A$, let $a \in A$ and $\eta>0$. Let $\delta>0$ such that $|f(b)-f(a)|<\eta$ whenever $b \in A$ and $d(a,b)<3\delta$. Now let $d(a,x)<\delta$. If $x \in A$ then $|F(f)(x)-F(f)(a)|=|f(x)-f(a)| < \eta$. If $x \in T\setminus A$ then, by Lemma \ref{lm:continuousextension}, $d(a,a_i)<3\delta$, whence $|f(a_i)-f(a)|<\eta$, for any $i$ satisfying $x \in U_i$ (of which there are finitely many). As $F(f)(x)$ is a convex combination of these $f(a_i)$, we obtain $|F(f)(x)-F(f)(a)|<\eta$. Hence $F(f)$ is continuous on $T$. Finally, it is clear that $F$ is linear.
	
	\begin{lemma}\label{lm:extensionmetric}
		Let $A\subseteq T$ be a non-empty closed set and let $\epsilon \in \left(0,\frac{1}{13}\right)$, $d\in \cM^T$, and $d_A\in \cM^A$. Suppose that $\rho_A (d_A, d\restrict_{A^2}) \leq \epsilon$ and $d(x,A) \leq \epsilon$ for every $x\in M$. Then there exists $\bar d \in \cM^T$ and a dual linear extension operator $E\colon \Lipo{A, d_A}\to \Lipo{T,\bar d}$, such that $\bar d$ extends $d_A$, $\rho_T(\bar d, d) \leq 13\epsilon$, and $\nm E = 1$.
	\end{lemma}
	\begin{proof}
		Pick a common base point in $A$ for both $\Lipo{A,d_A}$ and $\Lipo{T,d}$. 
		With the map $F$ above in hand, we define the pseudometric $\hat d\colon T^2\to [0,\infty)$ by $$\hat d(x,y) = \sup_{f\in B_{\Lipo{A,d_A}}} |F(f)(x)-F(f)(y)|$$(from the definition of $F$ it is evident that the supremum above is always finite). It is clear that $\hat d\restrict_{A^2} = d_A$. We will now prove that $|\hat d(x,y) - d(x,y)| \leq 11\epsilon$ for all $x,y\in T$. Given $x,y\in T$, without loss of generality we can assume $x\not\in A$. If $i$ is such that $x\in U_i$ then $d(x,x_i) \leq \frac{\epsilon}{3}$, and as $d(x_i,A) \leq \epsilon$, we have $d(x,a_i) \leq \frac{4}{3}\epsilon$. If $i$ and $j$ are such that $x\in U_i\cap U_j$ then $d(a_i,a_j) \leq d(a_i,x) + d(x,a_j) \leq \frac{8}{3}\epsilon$, and so for $f\in B_{\Lipo{A,d_A}}$, $$|f(a_i)-f(a_j)| \leq d_A(a_i,a_j) \leq d(a_i,a_j) + \epsilon \leq \frac{11}{3}\epsilon.$$ Pick some $i_0$ such that $x\in U_{i_0}$. As $F(f)(x)$ is a convex combination of the numbers in the set $\{f(a_i) : i\in \N, x\in U_i \}$, we have $|F(f)(x) - f(a_{i_0})| \leq \frac{11}{3}\epsilon$. If $y\not\in A$, we can similarly pick $j_0$ such that $d(y,a_{j_0}) \leq \frac{4}{3}\epsilon$ and $|F(f)(y)- f(a_{j_0})| \leq \frac{11}{3}\epsilon$ for any $f\in B_{\Lipo{A,d_A}}$. If $y\in A$ then the previous two inequalities are satisfied with $y$ substituted for $a_{j_0}$. We put $a'=a_{j_0}$ if $y\not\in A$ and $a' = y$ if $y\in A$. We have
		\begin{align*}
			d(y,a') \leq \frac{4}{3}\epsilon, \; \text{ and } \; |F(f)(y)- f(a')| \leq \frac{11}{3}\epsilon,
		\end{align*}
		and the same inequalities hold for $x$ and $a_{i_0}$ in place of $y$ and $a'$, respectively. Now if $g\in B_{\Lipo{A,d_A}}$ is such that $|g(a_{i_0}) - g(a')| = d_A(a_{i_0},a')$, then using the previous inequalities and the triangle inequality twice we have
		\begin{align*}
			\hat d(x,y) &\geq |F(g)(x)-F(g)(y)| \geq |g(a_{i_0}) - g(a')| - \frac{22}{3}\epsilon \\& = d_A(a_{i_0}, a') - \frac{22}{3}\epsilon \geq d(a_{i_0}, a') - \frac{25}{3}\epsilon \geq d(x,y) - 11\epsilon.
		\end{align*} 
		On the other hand, we have
		\begin{align*}
			\hat d(x,y) &= \sup_{f\in B_{\Lipo{A,d_A}}} |F(f)(x)-F(f)(y)| \leq \sup_{f\in B_{\Lipo{A,d_A}}} |f(a_{i_0})-f(a')| + \frac{22}{3}\epsilon \\&= d_A(a_{i_0}, a') + \frac{22}{3}\epsilon \leq d(a_{i_0}, a') + \frac{25}{3}\epsilon \leq d(x,y) + 11\epsilon.
		\end{align*}
		Therefore $|\hat d(x,y) - d(x,y)| \leq 11\epsilon$. From this we deduce that $\hat d$ is a pseudometric on $T$ which extends $d_A$.
		
		Now define the pseudometric $e\colon T^2\to [0,\infty)$ by $$e(x,y) = \sup\{|f(x)-f(y)| : f\in B_{\Lipo{T,d}}, f\restrict_A = 0\}.$$ Given $x,y\in T$, pick $a,b\in A$ such that $d(x,a), d(y,b) \leq \epsilon$, and note that $$e(x,y) \leq e(x,a) + e(a,b) + e(b,y) \leq d(x,a) + d(b,y) \leq 2\epsilon.$$ Therefore $\nm e _\infty \leq 2\epsilon$. We define $\bar d = \hat d + e$. It is easily checked that $\bar d$ extends $d_A$ and $\bar d(x,y) > 0$ whenever $x\not=y$, so $\bar d$ is a metric on $T$. We also clearly have $\rho_T(\bar d, d) \leq 13\epsilon$.
		
		We now check that $\bar d\in \cM^T$. First we will show that $\bar d\in C(T^2)$. Because $\bar d$ satisfies the triangle inequality, it is sufficient to show that if $(x_n)_n\subseteq T$ is such that $d(x_n, x)\to 0$ for some $x\in T$, then $\bar d(x_n,x)\to 0$. As $e \leq d$, all we need to show is that $\hat{d}(x_n,x) \to 0$. Suppose, for a contradiction, that there exist $(x_n)_n\subseteq T$, $x \in T$ and $\eta>0$ such that $d(x_n,x)\to 0$ and $\hat{d}(x_n,x) > \eta$ for all $n$. Then there exist functions $f_n\in B_{\Lipo{A,d_A}}$ such that $|F(f_n)(x_n)-F(f_n)(x)| > \eta$ for all $n$.
		
		At this point we must consider two cases. We address the case $x \in A$ first. Given that $d_A$ and $d\restrict_{A^2}$ are equivalent metrics on $A$, we can pick $\delta>0$ such that
		\[
		d_A(x,a) < \eta \quad\text{whenever }a \in A \text{ and }d(a,x) < \delta.
		\]
		Fix $n$ such that $d(x_n,x) < \frac{1}{3}\delta$. Either $x_n \in A$ or not. If $x_n \in A$ then
		\[
		|F(f_n)(x_n)-F(f_n)(x)| = |f_n(x_n)-f_n(x)| \leq d_A(x_n,x) < \eta,
		\]
		which isn't the case. Instead, suppose $x_n \in T\setminus A$. By \Cref{lm:continuousextension}, $d(x,a_i) < \delta$ whenever $x_n \in U_i$. Hence $|f_n(x)-f_n(a_i)| \leq d_A(x,a_i) < \eta$ for all such $i$, of which there are finitely many. As $F(f_n)(x_n)$ is a convex combination of these $f_n(a_i)$, we obtain $|F(f_n)(x_n)-F(f_n)(x)| < \eta$, which again violates our assumption.
		
		In the second case we assume $x \in T\setminus A$. By Alaoglu's theorem, $B_{\Lip_0(A,d_A)}$ is $w^*$-compact, and as $(A,d_A)$ is separable, $B_{\Lip_0(A,d_A)}$ is metrisable. The $w^*$-topology on bounded subsets of $\Lip_0(A,d_A)$ coincides with the topology of pointwise convergence, as mentioned just before \Cref{lm:dualcriterion}. Therefore, by taking a subsequence if necessary, we can assume that $f_n \to f \in B_{\Lip_0(A,d_A)}$ pointwise. Given that $(U_i)$ is locally finite, let $U \subseteq T\setminus A$ be a neighbourhood of $x$ that intersects just finitely many of the $U_i$, which we denote by $U_{k_1},\ldots,U_{k_m}$. We then have $F(g)(y) = \sum_{i=1}^m \lambda_{k_i}(y) g(a_i)$ for all $y\in U$ and $g\in C(A)$. This implies that $F(f_n) \to F(f)$ uniformly on $U$. For large enough $n$, $x_n \in U$, and
		\begin{align*}
			&\;|F(f_n)(x_n) - F(f_n)(x)| \\ \leq& \;|F(f_n)(x_n) - F(f)(x_n)| + |F(f)(x_n) - F(f)(x)| + |F(f)(x) - F(f_n)(x)| \\\leq& \; 2\nm{(F(f_n) - F(f))\restrict_U}_\infty + |F(f)(x_n) - F(f)(x)| \to 0,
		\end{align*} 
		as $n\to\infty$, a contradiction. Hence we conclude that $\bar d \in C(T^2)$.
		Then $\bar d\in \cM^*$, and so $\bar d\in \cM^T$ by \Cref{rmk:rhotmtdist}, since $\rho_T(\bar d, d) \leq 13\epsilon < 1$.
		
		From the inequality $\bar d \geq \hat d$ and the definition of $\hat d$, it follows that $F$ maps $\Lipo{A,d_A}$ into $\Lipo{T,\bar d}$, and $\Lip_{\bar d}(F(f))\leq \Lip_{d_A}(f)$ for $f\in \Lipo{A,d_A}$. Hence we can define $E\colon \Lipo{A,d_A} \to \Lipo{T,\bar d}$ by $E(f) = F(f)$. To show that $E$ is a dual operator, suppose that $(f_n)_n \subseteq B_{\Lipo{A,d_A}}$ converges to $f\in B_{\Lipo{A,d_A}}$ pointwise. Then from the definition of $F$ it is clear that $E(f_n)\to E(f)$ pointwise. By \Cref{lm:dualcriterion}, $E$ is a dual operator.
	\end{proof}
	
	The next result only makes partial use of \Cref{lm:extensionmetric}; the latter will be used fully in the proofs of \Cref{thm:mapdense,thm:failapdense}.
		
	\begin{proposition}\label{prop:non-sep}
		If $T$ is not compact then $\cM^T$ is non-separable in the topology of uniform convergence.
	\end{proposition}
		
	\begin{proof}
		As $(T,\alpha)$ is complete but not compact, there exists $\epsilon \in (0,\frac{1}{13})$ and a sequence $(x_n)_n \subseteq T$ (with $x_0$ the base point) such that $\alpha(x_n,x_m) \geq \epsilon$ for distinct $m,n \geq 0$. Define the closed sets $C=\{x_n \,:\, n \geq 0\}$ and
		\[
		A = C \cup \{x \in T \,:\, \alpha(x,C) \geq \epsilon\}.
		\]
		Given $E \subseteq C$ having at least two elements, we define a pseudometric $d'_E$ on $A$ by
		\[
		d'_E(x,y) = \begin{cases}
			0 & \text{if $x=y$ or $x,y \in A\setminus C$,}\\
			\epsilon &\text{if $x,y \in E$ are distinct,}\\
			\frac{1}{2}\epsilon &\text{otherwise.}
		\end{cases}
		\]
		The symmetry of $d'_E$ is immediate and the triangle inequality is straightforward to check. Also, it is clear that $d'_E \leq \alpha\restrict_{A^2}$. Therefore, the metric $\alpha\restrict_A+d'_E$ is Lipschitz equivalent to $\alpha\restrict_{A^2}$ on $A$; in particular it belongs to $\cM^A$. Thus we can apply \Cref{lm:extensionmetric} to obtain an extension $\bar{d}_E \in \cM^T$ of $\alpha\restrict_{A^2}+d'_E$. Now let $E,E' \subseteq C$ be distinct and have at least two elements. Without loss of generality there exists $x \in E\setminus E'$, and picking $y \in E\setminus\{x\}$ yields
		\[
		\rho_T(\bar{d}_E,\bar{d}_{E'}) \geq \min(1,|\bar{d}_E(x,y)-\bar{d}_{E'}(x,y)|) = \textstyle\frac{1}{2}\epsilon.
		\]
		Therefore, the uncountable family of all such metrics $\bar{d}_E$ is $\rho_T$-uniformly separated.
	\end{proof}
		
	We next state (a simplified version of) a result from \cite{talimpublished}, together with two remarks, that are needed in the proof of Theorem \ref{thm:failapdense}.
	
	\begin{lemma}[{cf. \cite{talimpublished}*{Corollary 2.3}}]\label{lem:talim}
		Let $K$ be homeomorphic to the Cantor set, $d,e \in \cM^K$ and $\epsilon>0$. Then there exists a partition $\{K_1,\ldots,K_m\}$ of clopen subsets of $K$ and a metric $d_K\in \cM^K$, such that
		\begin{align}\label{eq:dk2minusdkf}
			\nm{d_K - d}_\infty < \epsilon,
		\end{align}
		and $(K_i,d_K)$ is proportional to $(K,e)$ for all $i=1,\ldots,m$.
	\end{lemma}
	
	\begin{remark}
		We point out two facts about $\{K_1,\ldots,K_m\}$ and $d_K$ above, which follow ultimately from how they are constructed in \cite{talimpublished}*{Lemma 2.1}:
		\begin{enumerate}
			\item the partition $\{K_1,\ldots,K_m\}$ satisfies
			\begin{align}\label{eq:Dkileqepsilon2}
				D(K_i) < \frac{\epsilon}{2} \;\text{ for all } \; i=1,\ldots,m;
			\end{align}
			\item the metric $d_K$ satisfies
			\begin{align}\label{eq:dkxyequals}
				d_K(x,y) = D(K_i,K_j) \;\text{ whenever }\; x\in K_i, y\in K_j\text{ and } i\not=j.
			\end{align}
		\end{enumerate}
	\end{remark}
	
	We finish this section by defining a class of `flattening operators' for Lipschitz functions. The idea behind these operators is to replace a given Lipschitz function $f$ defined on a ball of a closed subset $A \subseteq T$ by a new one, defined on $A$, that agrees with $f$ on a (usually much) smaller ball of $A$, and vanishes outside the original ball, in such a way that the Lipschitz constant of the new function is not that much greater than that of $f$. These are analogous to the flattening operators used in \cite{smithtalim}*{Proposition 2.12} and are based on an idea found in the proof of \cite{godefroykalton}*{Proposition 5.1}. Fix constants $R>r>0$ and define $\mu=\mu_{r,R}\colon [0, \infty) \to [0,1]$ by
	\[
	\mu(t) =
	\begin{cases}
		1 & \text{if } t \leq r, \\
		\displaystyle \frac{\log(R/t)}{\log(R/r)} & \text{if } r < t < R, \\
			0 & \text{otherwise. } \\
	\end{cases}
	\]
	Now let $d\in\cM^T$ and let $A\subseteq T$ be a closed subset such that $x_0\in A$. Then define a linear operator $\Phi = \Phi^d_{A,r,R}\colon\Lipo{A\cap B_R^d, d} \to \Lipo{A, d}$ by
	\begin{equation}\label{eqn:flattening}
		\Phi(f)(x) =
		\begin{cases}
			\mu(d(x_0,x))f(x) & \text{if } x\in A\cap B_R^d, \\
			0 & \text{otherwise. } \\
		\end{cases}
	\end{equation}
	It is immediate that $\Phi(f)$ agrees with $f$ on $A\cap B_r^d$ and vanishes outside $A\cap B_R^d$.
		
	\begin{proposition}[{cf. \cite{smithtalim}*{Proposition 2.12}}]
		We have $\nm{\Phi} \leq 1 + \frac{1}{\log(R/r)}$.
	\end{proposition}
		
	\begin{proof}
		By elementary means, $|\mu(s)-\mu(t)| \leq |s-t|/(s\log(R/r))$ whenever $0 < s \leq t$. Given $f \in \Lipo{A\cap B_R^d, d}$, extend it to $g \in \Lipo{T,d}$ while preserving its Lipschitz constant by McShane's Extension Theorem \cite{weaver}*{Theorem 1.33}. Given $x,y \in A$ satisfying $s:=d(x_0,x) \leq d(x_0,y) =: t$,
		\begin{align*}
			|\Phi(f)(x)-\Phi(f)(y)| &\leq |\mu(s)-\mu(t)|\,|g(x)| + |\mu(t)|\,|g(x)-g(y)|\\
			&\leq \frac{\Lip(g)|s-t|}{\log(R/r)} + \Lip(g)d(x,y) \leq \left(\frac{1}{\log(R/r)} + 1\right)\Lip(f)d(x,y). \qedhere 
		\end{align*}
	\end{proof}
	
	\section{Proofs of the main results}\label{sec_main}
	
	We are now ready to give the proofs of the three main results.
	
	\begin{proof}[Proof of \Cref{thm:mapdense}]
		Let $\nu$ be a complete compatible metric on $\cM^T$. Let $d\in\cM^T$ and $\epsilon\in (0,1)$. We will inductively define an increasing sequence (with respect to inclusion) of countable subsets $(A_n)_{n=1}^\infty$, $A_n\subseteq T$, metrics $(d_n)_{n=1}^\infty \subseteq \cM^T$, and extension operators $(E_n)_{n=1}^\infty$, $E_n\colon\Lipo{A_n,d_n} \to \Lipo{T,d_n}$, such that $d_{n+1}$ and $d_n$ agree on $A_{n+1}$, $\nu(d_n, d_{n+1}) \leq \frac{\epsilon}{2^{n+1}}$, and $\nm{E_n} = 1$ for all $n\in\N$.
		
		It will help to start the induction at $n=0$ by setting $A_0=\{x_0\}$ and $d_0=d$. Given $A_n$ and $d_n$, we shall define $\epsilon_{n+1}$, $A_{n+1}$, $d_{n+1}$ and $E_{n+1}$. First, find $\epsilon_{n+1} \in (0,\min(\frac{1}{n+1}, \frac{1}{13}))$ such that $\nu(d',d_n) \leq\frac{\epsilon}{2^{n+1}}$ whenever $d'\in\cM^T$ and $\rho_T(d', d_n) \leq 13\epsilon_{n+1}$. Let $A_{n+1} \subseteq T$ be a countable set such that $A_n\subseteq A_{n+1}$, every compact subset of $A_{n+1}$ is finite and $d_n(x,A_{n+1}), d(x,A_{n+1}) \leq \epsilon_{n+1}$ for every $x\in T$. It is clear that $A_{n+1}$ is closed. We apply \Cref{lm:extensionmetric} to $A_{n+1}$, $\epsilon_{n+1}$, $d_n$ and $d_n\restrict_{A_{n+1}^2}$ to obtain a metric $d_{n+1} \in \cM^T$ and an extension operator $E_{n+1}\colon \Lipo{A_{n+1}, d_n\restrict_{A^2_{n+1}}} \to\Lipo{T,d_{n+1}}$ such that $d_{n+1}$ agrees with $d_n$ on $A_{n+1}$, $\rho_T(d_{n+1},d_n) \leq 13\epsilon_{n+1}$ and $\nm{E_{n+1}} = 1$. It follows that $\nu(d_{n+1},d_n) \leq \frac{\epsilon}{2^{n+1}}$. This completes the induction.
		
		As $d(x,A_n) \leq \epsilon_n \leq \frac{1}{n}$ for every $x\in T$ and $n\in\N$, $\bigcup_{n=1}^\infty A_n$ is dense in $T$. The sequence $(d_n)_n$ is $\nu$-Cauchy and hence converges uniformly to a metric $\bar d\in \cM^T$, which satisfies $\nu(\bar d, d) \leq \epsilon$. For $n\in\N$, define $G_n\colon\Lipo{A_n,\bar d}\to\Lipo{A_{n+1},\bar d}$ by $G_n(f) = E_n(f)\restrict_{A_{n+1}}$. As $d_k$ agrees with $d_n$ on $A_{n+1}$ whenever $k\geq n+1$, $\bar d$ agrees with both $d_{n+1}$ and $d_n$ on $A_{n+1}$. This implies that $\nm{G_n} = 1$. Now for $k\geq n+1$, define $H_n^k\colon \Lipo{A_n,\bar d}\to\Lipo{A_k,\bar d}$ by $H_n^k = G_{k-1}\circ\ldots\circ G_n$. Clearly $\nm{H_n^k} = 1$ and $H_n^k(f)$ agrees with $H_n^m (f)$ on $A_{\min(k,m)}$ whenever $m,k\geq n+1$ and $f\in\Lipo{A_n,\bar d}$. Finally, define $H_n\colon\Lipo{A_n,\bar d}\to\Lipo{T,\bar d}$ as follows. Pick $f\in\Lipo{A_n,\bar d}$. First, we define $H_n(f)$ on $\bigcup_{i=1}^\infty A_i$ by $H_n(f)(x) = H_n^k(f)(x)$, where $k = \max(n+1, i)$ and $i\in\N$ is such that $x\in A_i\setminus A_{i-1}$. Evidently $H_n(f)$ has Lipschitz constant at most $\Lip(f)$. We then extend $H_n(f)$ to $T$ by continuity and density of $\bigcup_{i=1}^\infty A_i$. It is clear that $H_n$ is an extension operator and $\nm{H_n}=1$.
		
		Now for every $n\in\N$ define the restriction operator $R_n\colon\Lipo{T,\bar d}\to\Lipo{A_n\cap B^{\bar d}_{n^2},\bar d}$ by $R_n(f) = f\restrict_{A_n\cap B^{\bar d}_{n^2}}$, and, having in mind \eqref{eqn:flattening}, the operator $P_n\colon\Lipo{T,\bar d} \to \Lipo{T,\bar d}$ by $$P_n = H_n\circ \Phi^{\bar d}_{A_n,n,n^2}\circ R_n.$$ Then $P_n$ is a finite-rank operator of norm $\nm{P_n} \leq 1+\frac{1}{\log n}$, such that $P_n(f)$ agrees with $f$ on $A_n\cap B^{\bar d}_n$ for every $f\in\Lipo{T,\bar d}$. 
		
		The operator $R_n$ is a dual operator by \Cref{lm:dualcriterion}, since if $(f_k)_k\subseteq\Lipo{T,\bar d}$ converges pointwise to $f\in \Lipo{T,\bar d}$, then clearly $R_n(f_k)\to R_n(f)$ pointwise. The map $H_n\circ \Phi_{A_n,n,n^2}$ is also a dual operator because its domain is a finite-dimensional space and thus automatically $w^*$-$w^*$ continuous. We thus obtain that $P_n$ is a dual operator. Now let $x\in T$ and $\xi > 0$. As $\bigcup_{n=1}^\infty A_n$ is dense in $T$, we can choose $n_0\in\N$ so that there exists some $y\in A_{n_0}\cap B^{\bar d}_{n_0}$ such that $\bar d(x,y) \leq \xi$. Then for any $n\geq n_0$ and $f\in B_{\Lipo{T,\bar d}}$, $P_n(f)(y) = f(y)$ and
		\begin{align*}
				|P_n(f)(x) - f(x)| &\leq |P_n(f)(x) - P_n(f)(y)| + |f(y)-f(x)|\\
				&\leq \left(2+\frac{1}{\log n}\right)\bar d(x,y) \leq \left(2+\frac{1}{\log n}\right)\xi.
		\end{align*}
		It follows that $P_n(f)(x)\to f(x)$ uniformly in $f\in B_{\Lipo{T,\bar d}}$. If $Q_n$ is the predual operator to $P_n$ for every $n\in\N$, we can easily conclude that for every $x\in T$, $\nm{Q_n(\delta_x) - \delta_x} \to 0$ as $n\to\infty$. As $\nm{Q_n} \to 1$, we have $\nm{Q_n(\nu) - \nu}\to 0$ for every $\nu\in \free{T,\bar d}$. Therefore $\free{T,\bar d}$ has the MAP and $\bar d\in \cA^{T,1}$. As $\nu(d,\bar d)\leq\epsilon$ and $d$ and $\epsilon$ were arbitrary, $\cA^{T,1}$ is dense in $\cM^T$.
	\end{proof}

	\begin{proof}[Proof of Theorem \ref{thm:failapdense}]
		Suppose $T$ is uncountable. Let $d\in \cM^T$ and $\epsilon \in \left(0,\frac{1}{13}\right)$ be arbitrary. As $T$ is completely metrisable and separable, it is a Polish space. By \cite{kechris}*{Theorem 6.4} $T$ contains a non-empty perfect subset, as $T$ is uncountable, and by \cite{kechris}*{Theorem 6.2}, $T$ contains a subset $K$ homeomorphic to the Cantor set. By \cite{hajeklancienpernecka}*{Corollary 2.2} there exists a metric $e_K^f\in \cA^K_f$. We apply \Cref{lem:talim} to $K$ and $d\restrict_{K^2},e_K^f\in \cM^K$ to obtain a partition $\{K_1,\ldots,K_m\}$ of clopen subsets of $K$ and a metric $d_K\in \cM^K$ satisfying \eqref{eq:dk2minusdkf} -- \eqref{eq:dkxyequals}, and such that $(K_i,d_K)$ is proportional to $(K,e_K^f)$ for all $i=1,\ldots,m$. As proportional metric spaces have isometrically isomorphic free spaces, $\free{K_1,d_K}$ fails the AP. By \Cref{lm:closedboundedcompl}, $\free{K_1,d_K}$ is complemented in $\free{K,d_K}$. Hence $\free{K,d_K}$ fails the AP. 
		
		Let $C = K \cup \{x\in T : d(x,K) \geq \epsilon\}$. Obviously $C$ is closed. Now define $d_C\colon C^2\to [0,\infty)$ by
		\begin{equation*}
			d_C(x,y) = 
			\begin{cases}
				d_K(x,y) & \text{ if } x,y\in K \\
				D(x,K_i)+\frac{\epsilon}{2} & \text{ if } x\in C\setminus K, \text{ and } y\in K_i, \\
				D(y,K_i) + \frac{\epsilon}{2} & \text{ if } x\in K_i, \text{ and } y\in C\setminus K, \\
				d(x,y) & \text{ if } x,y\in C\setminus K.
			\end{cases}
		\end{equation*}
		We prove that $d_C\in \cM^C$. Obviously symmetry holds and $d_C(x,y)=0$ if and only if $x=y$. To show the triangle inequality, pick $x,y,z\in C$. If $x,y,z\in K$ or $x,y,z\in C\setminus K$ then the inequality is obvious. Suppose $x,y\in K$ and $z\in C\setminus K$. Then, using \eqref{eq:dk2minusdkf}, $$d_C(x,y) = d_K(x,y) < d(x,y) + \epsilon \leq d(x,z) + d(z,y) + \epsilon \leq d_C(x,z) + d_C(z,y).$$ If $x,y\in K_i$ for some $i$ then trivially $d_C(x,z) \leq d_C(x,y) + d_C(y,z)$ and likewise $d_C(y,z) \leq d_C(y,x) + d_C(x,z)$. Suppose $x\in K_i$ and $y\in K_j$ where $i\not=j$. Let $x'\in K_i$ be such that $D(z,K_i) = d(z,x')$. Then
		\begin{align*}
			d_C(x,z) &= d(z,x') + \frac{\epsilon}{2} \leq d(z,y) + d(y,x') + \frac{\epsilon}{2} \leq d(y,x') + d_C(y,z) \\&\leq D(K_i,K_j) + d_C(y,z) = d_K(x,y) + d_C(y,z) = d_C(x,y) + d_C(y,z),
		\end{align*}
		where the penultimate equality follows from \eqref{eq:dkxyequals}. The inequality $d_C(y,z) \leq d_C(y,x) + d_C(x,z)$ can be shown similarly.
		
		Suppose now that $x\in K$ and $y,z\in C\setminus K$. If $x\in K_i$, let $x'\in K_i$ be such that $D(y,K_i) = d(y,x')$. We have $$d_C(x,y) = d(y,x') + \frac{\epsilon}{2} \leq d(y,z) + d(z,x') + \frac{\epsilon}{2} \leq d_C(x,z) + d_C(z,y).$$ We can show similarly that $d_C(x,z) \leq d_C(x,y) + d_C(y,z)$. Finally, $$d_C(y,z) = d(y,z) \leq d(y,x) + d(x,z) \leq d_C(y,x) + d_C(x,z).$$ Therefore $d_C$ is a metric on $C$. It is also clear that $d_C$ is compatible with the topology on $C$, since $d_C\restrict_{K^2}$ and $d_C\restrict_{(C\setminus K)^2}$ are compatible with the topologies on $K$ and $C\setminus K$, respectively, $d(K,C\setminus K) > 0$ and $d_C(K,C\setminus K) \geq \frac{\epsilon}{2}$. As $d_C$ coincides with $d$ on $C\setminus K$ and $K$ is compact, it is easy to see that $d_C$ is proper. Therefore $d_C \in \cM^C$.
		
		We will now show that $\rho_C(d_C, d\restrict_{C^2}) \leq \epsilon$. Indeed, if $x\in C\setminus K$ and $y\in K_i$ for some $i$, let $y'\in K_i$ be such that $D(x,K_i) = d(x,y')$. Then \[|d_C(x,y) - d(x,y)| = \left|d(x,y') -d(x,y) + \frac{\epsilon}{2}\right| \leq d(y,y') + \frac{\epsilon}{2} \leq D(K_i) + \frac{\epsilon}{2} < \epsilon \tag*{by \eqref{eq:Dkileqepsilon2}.}\] The inequalities for $x,y\in K$ and $x,y\in C\setminus K$ follow from \eqref{eq:dk2minusdkf} and the definition of $d_C$, respectively.
		
		By \Cref{lm:extensionmetric}, there exists $\bar d\in \cM^T$ and a dual linear extension operator $E\colon\Lipo{C,d_C} \to \Lipo{T,\bar d}$ such that $\bar d$ extends $d_C$, $\rho_T(\bar d, d) \leq 13\epsilon$ and $\nm E = 1$. The predual of $E$ is a projection of $\free{T,\bar d}$ onto $\free{C,d_C}$. By \Cref{lm:closedboundedcompl}, $\free{K,d_K}$ is complemented in $\free{C,d_C}$. Hence $\free{C,d_C}$, and consequently $\free{T,\bar d}$, fail the AP. As $\rho_T(\bar d, d) \leq 13\epsilon$ and $d$ and $\epsilon$ were arbitrary, we obtain that $\cA_f^T$ is dense in $\cM^T$.
	\end{proof}

	\begin{proof}[Proof of \Cref{prop:residual}]
		Being zero-dimensional (and locally compact), $T$ has a basis of compact open sets, thus for every $n\in\N$ there is a compact open subset $T_n\subseteq T$ such that $B_n^{\alpha} \subseteq T_n$. For each $n,k\in\N$, let $\cP^k_n$ be a partition of $T_n$ into compact open subsets of $\alpha$-diameter less than $\frac{1}{k}$, and define $$U_n^k = \left\{d\in\cM^T \; : \; \chi_n^k(d) := \max_{C,C'\in \cP^k_n, C \neq C'} \frac{D(C,C')}{d(C, C')} < 1 + \frac{1}{k}\right\}.$$ Given distinct $C,C'\in \cP^k_n$, the assignment $d \mapsto \frac{D(C,C')}{d(C, C')}$ is continuous. Therefore $U_n^k$ is open. Fix $n,k_0\in\N$. We will prove that $\bigcup_{k=k_0}^\infty U_n^k$ is dense in $\cM^T$. Pick $d\in\cM^T$ and $\epsilon \in (0,1)$. Choose $k\geq k_0$ so that the $d$-diameter of the elements of $\cP_n^k$ is less than $\frac{\epsilon}{2}$. Let $C_1,\ldots,C_p$ be the elements of $\cP_n^k$. Define $\bar d\colon T^2\to [0,\infty)$ by
		\begin{equation*}
			\bar d(x,y) = 
			\begin{cases}
				d(x,y) & \text{ if } x,y\in C_i \text{ for some } i \text{, or if } x, y \in T\setminus T_n, \\
				D(C_i,C_j) & \text{ if } x\in C_i, y\in C_j \text{ and } i\not=j, \\
				D(x,C_i) & \text{ if } y\in C_i\text{ for some } i \text{ and } x\in T\setminus T_n, \\
				D(y,C_i) & \text{ if } x\in C_i \text{ for some } i \text{ and } y\in T\setminus T_n. \\ 
			\end{cases}
		\end{equation*}
		As in the proof of \cite{talimpublished}*{Lemma 2.1}, we can show that $\bar d$ is a metric, $\rho_T(\bar d, d) < \epsilon$, and $\bar d\restrict_{T_n^2}$ is compatible with the topology of $T_n$. Clearly, $\bar d|_{(T\setminus T_n)^2}$ is also compatible with the topology on $T\setminus T_n$. As $\bar d \geq d$, $\bar d(T_n, T\setminus T_n) \geq d(T_n, T\setminus T_n) > 0$, and so both $T_n$ and $T\setminus T_n$ are $\bar d$-open. Therefore $\bar d$ is compatible with the topology of $T$. As $\bar d$ coincides with $d$ on $T\setminus T_n$, the former is also proper and thus belongs to $\cM^T$. Since $\bar d(C_i,C_j) = \bar D(C_i,C_j)$, $\bar d\in U_n^k$. Therefore $\bigcup_{i=k_0}^\infty U_n^i$ is dense as claimed. 
		
		To finish the proof, it suffices to show that 
		\begin{align}\label{eq:inclusion}
			\bigcap_{n=1}^\infty \bigcup_{k=n}^\infty U_n^k \subseteq \cA^{T,1}.
		\end{align} 
		Pick $d$ in the set on the left-hand side of \eqref{eq:inclusion}. Then there is a sequence of numbers $(k_n)_n\subseteq \N$ such that $k_n\geq n$ and $d\in U^{k_n}_n$. For each $n\in\N$, suppose $\cP_n^{k_n} = \{ C_n^1,\ldots,C_n^{p_n} \}$, and pick points $x_n^i\in C_n^i$, $i=1,\ldots,p_n$ (put $x_n^i = x_0$ whenever $x_0\in C_n^i$ for some $i$). Define the finite-rank operator $S_n\colon\Lipo{T_n,d}\to\Lipo{T_n,d}$ by $T_n(f)(x) = f(x_n^i)$ whenever $x\in C_n^i$, $i=1,\ldots,p_n$. As in the proof of \cite{talimpublished}*{Lemma 3.4} we can show that
		\[
		\nm{S_n} \leq \chi_n^{k_n}(d) \leq 1 + \frac{1}{k_n} \leq 1 + \frac{1}{n}.
		\]
		Suppose $x\in T$. For all sufficiently large $n$, let $i_n\in\N$ be such that $x\in C_n^{i_n}$. As $\alpha$-$\diam(C_n^{i_n}) \to 0$ as $n\to\infty$ and all $C_n^{i_n}$ are contained in some fixed compact subset of $T$, we have $D(C_n^{i_n})\to 0$. As $S_n(f)(x_n^{i_n}) = f(x_n^{i_n})$ for all $f\in\Lipo{T_n,d}$, we can show that $|S_n(f)(x) - f(x)| \to 0$ as $n\to\infty$ uniformly in $f\in B_{\Lipo{T,d}}$ in the same way as at the end of the proof of \Cref{thm:mapdense}.
		
		Let $(m_n)_n\subseteq \N$ be an increasing sequence such that $B_{n^2}^d\subseteq B_{m_n}^\alpha$. Then $B_{n^2}^d\subseteq T_{m_n}$ for every $n\in\N$. Define the restriction operators $R_n\colon\Lipo{T,d}\to\Lipo{T_n,d}$, $R_n(f) = f\restrict_{T_n}$, and $\bar R_n\colon\Lipo{T_{m_n},d}\to\Lipo{B_{n^2}^d ,d}$, $\bar R_n(f) = f\restrict_{B_{n^2}^d}$. Finally, again recalling \eqref{eqn:flattening}, define $P_n\colon\Lipo{T,d}\to\Lipo{T,d}$ by $$P_n = \Phi^d_{T,n,n^2} \circ \bar R_n \circ S_{m_n}\circ R_{m_n}.$$ It is clear that $P_n$ is a finite-rank operator such that $\nm{P_n} \leq (1+\frac{1}{\log n})(1+\frac{1}{m_n})$ and that $P_n(f)$ agrees with $S_{m_n}(f\restrict_{T_{m_n}})$ on $B^d_n$ for every $f\in\Lipo{T,d}$. Therefore $|P_n(f)(x) - f(x)| \to 0$ as $n\to\infty$ uniformly in $f\in B_{\Lipo{T,d}}$. As shown in the last paragraph of the proof of \Cref{thm:mapdense}, we can show that $P_n$ is a dual operator, since $R_{n^2}$ is a restriction operator and $S_{n^2}$ is a composition of a restriction operator and an operator whose domain is a finite-dimensional space. As $(m_n)_n$ is increasing, $\nm{P_n}\to 1$. Again, as at the end of the proof of \Cref{thm:mapdense}, the sequence of preduals to the $P_n$ implies that $\free{T,d}$ has the MAP. Therefore $d\in\cA^{T,1}$, and so the inclusion \eqref{eq:inclusion} holds.
	\end{proof}
	
	We finish by posing an open problem that arises naturally from this work.
		
		\begin{problem}
			If $T$ is uncountable and not zero-dimensional, what more can be said about the topological nature of the sets $\cA^{T,1}$ and $\cA_f^T$? For example, is one of these sets residual? We do not know the answer to this even in the compact case.
	\end{problem}
	
	\bibliography{document}
	
\end{document}